\documentclass{article}
\usepackage{maa-monthly}
\usepackage[final]
{showkeys}
\usepackage{graphicx}
\usepackage{enumerate}
\usepackage{multirow}
\usepackage[skip=40pt,font=large]{caption}
\usepackage{framed}

\usepackage[framemethod=tikz]{mdframed}

\usepackage{trimclip}

\usepackage{needspace}

\usepackage{hyperref}

\theoremstyle{theorem}
\newtheorem{theorem}{Theorem}
\newtheorem{proposition}[theorem]{Proposition}

\newtheorem{lemma}[theorem]
{Lemma}

\theoremstyle{definition}

\newtheorem{remark}[theorem]
{Remark}

\theoremstyle{definition} 
\newtheorem*{remark*}{Remark}

\renewcommand{\skip}[1]{}

\newcommand{\N}{\mathbb{N}}
\newcommand{\R}{\mathbb{R}}

\newcommand{\E}{\operatorname{\mathsf{E}}}
\renewcommand{\P}{\operatorname{\mathsf{P}}}

\newcommand{\1}{\operatorname{\mathsf{1}}}

\newcommand{\al}{\alpha}

\renewcommand{\O}{\underset{\ffrown}{<}}
\newcommand{\OG}{\underset{\ffrown}{>}}

\makeatletter
\DeclareRobustCommand{\lasymp}{\lg@asymp{<}}
\DeclareRobustCommand{\gasymp}{\lg@asymp{>}}

\newcommand{\under@asymp}[1]{\clipbox{0pt 0pt 0pt {0.5\height}}{$\m@th#1\asymp$}}
\newcommand{\lg@asymp}[1]{\mathrel{\mathpalette\lg@asymp@{#1}}}
\newcommand{\lg@asymp@}[2]{%
  \vcenter{%
    \offinterlineskip
    \m@th
    \ialign{%
      \hfil##\hfil\cr
      $#1#2$\cr
      \under@asymp{#1}\cr
    }%
  }%
}
\makeatother

\renewcommand{\O}{\lasymp}
\renewcommand{\OG}{\gasymp}

\makeatletter

\long\def\@maketablecaption#1#2{%
    \rule[.1\baselineskip]{0pt}{\baselineskip}
    \small\textbf{#1.}\enspace #2\strut
    \par
  \vskip2pt}
\makeatother  

\begin{document}

\title{
Large Deviations Of Sums 
Mainly Due To Just One Summand}
\markright{
Large Deviations 
Mainly Due To Just One Summand}
\author{Iosif Pinelis
}





\maketitle

\begin{abstract} 
We present a formalization of the well-known thesis that, in the case of independent identically distributed random 
variables $X_1,\dots,X_n$ with power-like tails of index $\al\in(0,2)$,
large deviations of the sum $X_1+\dots+X_n$ are 
primarily due to just one of the summands.   
\end{abstract}






\section{Introduction, summary, and discussion.}\label{intro}

Let $X_1,X_2,\dots$ be \break independent identically distributed 
random variables. For each natural $n$, let \break 
$S_n:=\sum_1^n X_i$. 

Heyde \cite{heyde68} showed the following: Suppose that, for some sequence $(B_n)$ of positive real numbers, $S_n/B_n$ converges in distribution to a stable law of index \break  $\al\in(0,2)\setminus\{1\}$, whose support is the entire real line $\R$. 
(
For a definition and basic properties of stable laws, see e.g.\ \cite[\S IV.3]{pet75}.)
Then, for any sequence $(x_n)$ going \break 
to $\infty$, 
\begin{equation}\label{eq:heyde}
	\P(|S_n|>x_nB_n)\sim\P(\max_{1\le i\le n}|X_i|>x_nB_n). 
\end{equation}
%

As indicated in \cite{heyde68}, one-sided analogs of \eqref{eq:heyde} could also be obtained, even in the case $\al=1$. However, such a task would involve additional technical difficulties.
 
%
%
The conditions in \cite{heyde68} for \eqref{eq:heyde} imply that the tail of the distribution of each $X_i$ is power-like -- more specifically, 
\begin{equation}\label{eq:powertails}
	\P(|X_1|>u)=u^{-\al+o(1)}\text{\quad as\quad $u\to\infty$. }
\end{equation}

This work by Heyde was followed by a large number of publications, including 
\cite{nagaev.a.69I,nagaev.a.69II,
pin85,borovkov03}. 

The 
asymptotic equivalence \eqref{eq:heyde} and, especially, its proof suggest the well-known interpretation that, in the cases of power-like tails as in \eqref{eq:powertails}, large deviations of the sum $S_n$ are 
mainly due to just one of the summands $X_1,\dots,X_n$. 


In this note, we present a formal version of this interpretation:   
\begin{theorem}\label{th:}
Take any $\al\in(0,2)$. Let $X_1,X_2,\dots$ and $S_n$ be as in the first paragraph of this note. 
To avoid technicalities, suppose that the distribution of $X_1$ is symmetric about $0$ and has a probability density function $f$ such that 
\begin{equation}\label{eq:f}
	f(u)
	\asymp u^{-1-\al} \text{\quad as\quad $u\to\infty$ }
\end{equation}
(cf.\ \eqref{eq:powertails}). 
Then 
\begin{equation}\label{eq:sim}
	\P(S_n>x)\sim
	\P\Big(S_n>x,\,\bigcup_{i\in[n]}\big\{X_i>x,|S_n-X_i|\le bx,\,\max_{j\in[n]\setminus\{i\}}|X_j|\le cx\big\}\Big)
\end{equation}
whenever $n\in\N$, $x\in(0,\infty)$, $c\in(0,1)$, and $b\in(0,1)$ vary in such a way that 
\begin{align}
	n&<<x^\al, \label{eq:LD} \\ 
	nx^{-\al}&<<c^{2\al}, \label{eq:c} \\ 
nx^{-\al}&<<	b^2 c^{\al-2}.  \label{eq:b}
\end{align}
\end{theorem}

Here, as usual, $\N:=\{1,2,\dots\}$ and $[n]:=\{1,\dots,n\}$ for $n\in\N$.  
%
For positive expressions $E$ and $F$ (in terms of $x,n,c,b$), 
we write (i) $E\sim F$ if $E/F\to1$; \break 
(ii) $E<<F$ or, equivalently, $F>>E$ if $E=o(F)$
---that is, if $E/F\to0$; \break 
(iii) $E\O F$ or, equivalently, $F\OG E$ if 
$\limsup E/F<\infty$; and (iv) $E\asymp F$ if $E\O F\O E$. The ``much smaller than'' sign $<<$ should not be confused with Vinogradov's symbol $\ll$ 
(the latter is usually taken to mean the same as $\O$). 

\begin{proposition}\label{prop:}
For $S_n$ as in Theorem~\ref{th:} and for all $n\in\N$ and $x\in(0,\infty)$, we have $\P(S_n>x)\to0$ if and only if condition \eqref{eq:LD} holds. 
Moreover, if either \eqref{eq:LD} holds or $\P(S_n>x)\to0$, then $\P(S_n>x)\asymp nx^{-\al}$. 
\end{proposition}

\begin{remark}\label{rem:all LD}
Condition $\P(S_n>x)\to0$ means precisely that $\P(S_n>x)$ is a large-deviation probability for $S_n$. So, in view of Proposition~\ref{prop:}, Theorem~\ref{th:} concerns all the large deviations of $S_n$. 
\qed
\end{remark}

\begin{remark}\label{rem:p,tp} 
Given \eqref{eq:c}, for \eqref{eq:b} to hold it is enough that $b\asymp c$ or even $b\OG c^{1+\al/2}$.  
Therefore and because the probability on the right-hand side of \eqref{eq:sim} is non-decreasing in $c$ and in $b$, without loss of generality 
\begin{equation}\label{eq:c<<1}
	c<<1 \quad\text{and}\quad b<<1. 
\end{equation}
So, \eqref{eq:sim} shows 
that the large deviation event $\{S_n>x\}$ is mainly due to just one of the summands $X_1,\dots,X_n$. More specifically, \eqref{eq:sim} tells us that, given $S_n>x$, the conditional probability that 
exactly one of the $X_i$'s is $>x$ while the absolute values of the other $X_i$'s and of the sum of the other $X_i$'s are all $o(x)$ is close to $1$. \qed
\end{remark}

\begin{remark}\label{rem:n}
In contrast with \eqref{eq:heyde}, the condition $n\to\infty$ is not required in Theorem~\ref{th:}; in particular, $n$ may be fixed there. However, 
it is clear that condition \eqref{eq:LD} in Theorem~\ref{th:} 
necessarily implies that 
$x\to\infty$. 
In another distinction from \eqref{eq:heyde}, in Theorem~\ref{th:} the common distribution of the $X_i$'s is not required to be in the domain of attraction of a stable law. \qed
\end{remark}

\section
{Proofs.}\label{proofs}
\begin{proof}[Proof of Theorem~1]
This proof is based on two lemmas. To state the lemmas, let us introduce the following notations: 
\begin{align}
	p_0(n,x):=&
	\P\Big(S_n>x,\,\max_{j\in[n]}|X_j|\le cx\Big), \label{eq:p0} \\ 
	p_{\ge2}(n,x):=&
	\P\Big(S_n>x,\,\bigcup_{i\in[n]}\,\bigcup_{j\in[n]\setminus\{i\}}\{|X_i|>cx,|X_j|>cx\}\Big),  
\label{eq:p2} \\ 
	p_{1,0}(n,x):=&
	\P\Big(S_n>x,\,\bigcup_{i\in[n]}\big\{cx<|X_i|\le x,\,\max_{j\in[n]\setminus\{i\}}|X_j|\le cx\big\}\Big), \label{eq:p10} \\ 
	p_{1,1,-}(n,x):=
&\P\Big(S_n>x,\,\bigcup_{i\in[n]}\big\{X_i<-x,\,\max_{j\in[n]\setminus\{i\}}|X_j|\le cx\big\}\Big), \label{eq:p11-} \\  
	p_{1,1,+}(n,x)
	:=&\P\Big(S_n>x,\,\bigcup_{i\in[n]}\big\{X_i>x,\max_{j\in[n]\setminus\{i\}}|X_j|\le cx\big\}\Big).  \label{eq:tp} 
\end{align}

\begin{lemma}\label{lem:>} 
For $n$ and $x$ as in the conditions of Theorem~\ref{th:} 
(that is, for $n\in\N$ and $x\in(0,\infty)$ such that \eqref{eq:LD} holds), we have 
\begin{equation}\label{eq:>}
	\P(S_n>x)\OG
	n\P(X_1>x)\asymp
	nx^{-\al}. 
\end{equation}
\end{lemma}

\begin{proof}
By \eqref{eq:f}, 
\begin{equation}\label{eq:p1}
	\P(X_1>u)\asymp u^{-\al}\text{\quad as\quad $u\to\infty$.} 
\end{equation}
So, in view of 
\eqref{eq:LD}, $n\P(X_1>x)\asymp nx^{-\al}<<1$. Now \eqref{eq:>} follows from \cite[inequality~V, (5.10)]{feller_vol2}, 
which immediately 
implies $\P(S_n>x) 
\ge
\tfrac14\,(1-e^{-2n\P(X_1>x)})$ (since the distribution of $X_1$ is symmetric and absolutely continuous). 
\end{proof}

\begin{lemma}\label{lem:<<}
For $n$, $x$, and $c$ as in the conditions of Theorem~\ref{th:}, 
\begin{align}
	p_0(n,x)&<<nx^{-\al}, \label{eq:0} \\ 
	p_{\ge2}(n,x)&<<nx^{-\al}, \label{eq:ge2} \\ 
	p_{1,0}(n,x)&<<nx^{-\al}, \label{eq:1,0} \\ 
	p_{1,1,-}(n,x)&<<nx^{-\al}. \label{eq:1,1,-} 
\end{align}
\end{lemma}

\begin{proof}
For all natural $i$, let  
\begin{equation*}
	Y_i:=X_i\,\1(|X_i|\le cx), 
\end{equation*}
where $\1(A)$ denotes the indicator of an assertion $A$, so that $\1(A)=1$ if $A$ is true, and $\1(A)=0$ if $A$ is false. Then the $Y_i$'s are independent identically distributed 
symmetric random variables. 
Also, by \eqref{eq:c} 
and 
\eqref{eq:LD}, 
$(cx)^{2\al}>>nx^{\al}
>>1$, so that \break 
$cx>>1$. 
Therefore, in view of \eqref{eq:f}, for some real $A>0$ we have 
\begin{equation*}
	\E Y_1^2
	\O\int_0^A u^2 f(u)\,du+\int_A^{cx}u^2 u^{-1-\al}\,du\asymp(cx)^{2-\al}. 
\end{equation*} 
Therefore, with 
\begin{equation*}
T_n:=\sum_1^n Y_i,	
\end{equation*}
by \eqref{eq:p0}, 
Markov's inequality, and \eqref{eq:c<<1},  
\begin{equation}\label{eq:p0<}
	p_0(n,x)\le\P(T_n>x)\le\frac{\E T_n^2}{x^2}=\frac{n\E Y_1^2}{x^2}
	\O c^{2-\al}\frac n{x^\al}<<
	nx^{-\al}.   
\end{equation}
So, \eqref{eq:0} is proved. 

Next, by \eqref{eq:p2}, \eqref{eq:p1}, 
and \eqref{eq:c}, 
\begin{equation*}
	p_{\ge2}(n,x)
	\le\binom n2\P(|X_1|>cx,|X_2|>cx) 
	\O n^2 (cx)^{-2\al}
	<<nx^{-\al}, 
\end{equation*}
which proves \eqref{eq:ge2}. 

Further, using \eqref{eq:p10}, \eqref{eq:f}, and Markov's inequality as in \eqref{eq:p0<}, we have 
\begin{equation}\label{eq:p_{1,0}=}
\begin{aligned}
	p_{1,0}(n,x)&=n\P(S_n>x,\ cx<|X_1|\le x, |X_2|\le cx,\dots,|X_n|\le cx) \\ 
&\le	n\P(cx<|X_1|\le x, Y_2+\cdots+Y_n>x-X_1) \\ 
&\asymp n\int_{cx}^x 
u^{-1-\al}\P(Y_2+\cdots+Y_n>x-u)\,du
\O I,
\end{aligned} 
\end{equation}
where 
\begin{equation*}
	I:=\int_{cx}^x 
	g(u)\,du,
	\quad g(u):=nu^{-1-\al}\min\Big(1,\frac{n(cx)^{2-\al}}{(x-u)^2}\Big). 
\end{equation*}
Next, 
\begin{equation}\label{eq:u_x}
	u_x:=x-n^{1/2}(cx)^{1-\al/2}
	\sim x,
\end{equation}
by 
conditions \eqref{eq:c} and \eqref{eq:c<<1} on $c$. 
It follows that 
\begin{equation*}
I=I_1+I_2+I_3, 
\end{equation*}
where 
\begin{multline*}
	I_1:=\int_{cx}^{x/2} 
	g(u)\,du
	\le\int_{cx}^\infty 
	nu^{-1-\al}\frac{n(cx)^{2-\al}}{(x/2)^2}\,du 
	\asymp\Big(\frac n{x^\al}\Big)^2 c^{2-2\al} 
	<<nx^{-\al},
\end{multline*}
again by 
the mentioned conditions on $c$;  
\begin{multline*}
	I_2:=\int_{x/2}^{u_x} 
	g(u)\,du
	\le\int_{-\infty}^{u_x} 
	n(x/2)^{-1-\al}\frac{n(cx)^{2-\al}}{(x-u)^2}\,du \\ 
	\asymp\Big(\frac n{x^\al}\Big)^{3/2} c^{1-\al/2} 
	<<nx^{-\al},
\end{multline*}
once again by 
the conditions on $c$; and, in view of 
the definition of $u_x$ in \eqref{eq:u_x}, 
\begin{multline*}
	I_3:=\int_{u_x}^x 
	g(u)\,du\le(x-u_x)nu_x^{-1-\al}
	\asymp\Big(\frac n{x^\al}\Big)^{3/2} c^{1-\al/2} 
	<<
	nx^{-\al},
\end{multline*}
as in 
the bounding of $I_2$. 
%
So, 
the bound on $p_{1,0}(n,x)$ in \eqref{eq:1,0} follows immediately from \eqref{eq:p_{1,0}=} and the bounds on the integrals $I_1,I_2,I_3$.

Finally,  
in view of 
the definition of $p_{1,1,-}(n,x)$ in \eqref{eq:p11-}, 
\begin{align}
	p_{1,1,-}(n,x)&=n\P\Big(S_n>x,\,X_1<-x,\,\max_{j\in[n]\setminus\{1\}}|X_j|\le cx\big\}\Big) \notag \\ 
	&\le n\P\Big(S_n-X_1>x,\,X_1<-x,\,\max_{j\in[n]\setminus\{1\}}|X_j|\le cx\big\}\Big) \notag
	\\ 
	&\le n\P(T_n-Y_1>x,\,X_1<-x) \notag \\ 
	&=\P(T_n-Y_1>x)\,n\P(X_1<-x) 
	\O\frac{nc^{2-\al}}{x^\al}\frac n{x^\al}<<nx^{-\al}. 
	\notag
\end{align}
\newpage
\noindent
\nopagebreak 
The 
$\O$ comparison 
here is obtained by bounding $\P(T_n-Y_1>x)$ similarly to the bounding of $\P(T_n>x)$ in \eqref{eq:p0<} and using 
the symmetry of the distribution of $X_1$, 
the condition $x\to\infty$, and the relation \eqref{eq:p1}; 
the $<<$ comparison in 
the above multiline display follows, 
yet again, by 
the conditions on $c$. 
So, \eqref{eq:1,1,-} is proved as well. 

This completes the proof of Lemma~\ref{lem:<<}. 
\end{proof}


Now we can complete the proof of Theorem~\ref{th:}. 
Note that 
\begin{equation*}
	\P(S_n>x)=p_0(n,x)+p_{\ge2}(n,x)+p_{1,0}(n,x)+p_{1,1,-}(n,x)+p_{1,1,+}(n,x). 
\end{equation*}
So, by 
Lemmas~\ref{lem:<<} and \ref{lem:>}, 
\begin{equation}\label{eq:sim p_{1,1,+}(n,x)}
	\P(S_n>x)\sim
	p_{1,1,+}(n,x).
\end{equation}

Finally, 
the difference between $
p_{1,1,+}(n,x)$ and the probability on the right-hand side of \eqref{eq:sim} is 
%
\begin{align}
	&\le n\P\Big(X_1>x,|S_n-X_1|> bx,\,\max_{j\in[n]\setminus\{1\}}|X_j|\le cx\big\}\Big) \notag \\ 
	&\le n\P(X_1>x)\P(|T_n-Y_1|> bx) \notag \\ 
	&\O \P(S_n>x)\frac{n(cx)^{2-\al}}{(bx)^2}
	<<\P(S_n>x); 
	\notag
\end{align}
the 
$\O$ comparison 
here is obtained using the $\OG$ comparison in \eqref{eq:>} and 
bounding $\P(T_n-Y_1>bx)$ similarly to the bounding of $\P(T_n>x)$ in \eqref{eq:p0<}; 
and the 
latter $<<$ comparison 
follows by \eqref{eq:b}. 
%
Now \eqref{eq:sim} follows from \eqref{eq:sim p_{1,1,+}(n,x)}. 

The proof of Theorem~\ref{th:} is complete. 
\end{proof}

\nopagebreak
\needspace{48pt}
\nopagebreak
\begin{proof}[Proof of Proposition~\ref{prop:}]
Suppose first that condition \eqref{eq:LD} holds. Then, by \eqref{eq:sim p_{1,1,+}(n,x)}, \eqref{eq:tp}, 
and \eqref{eq:>}, 
$\P(S_n>x)\sim	p_{1,1,+}(n,x)\le n\P(X_1>x)\asymp nx^{-\al}
\O\P(S_n>x)$, 
so that 
$\P(S_n>x)\asymp nx^{-\al}\to0$. 

On the other hand, if $\P(S_n>x)\to0$, then, by the inequality  $\P(S_n>x)\ge\tfrac14\,(1-e^{-2n\P(X_1>x)})$ in the proof of Lemma~\ref{lem:>}, we have $n\P(X_1>x)\to0$, and hence $\P(X_1>x)\to0$ and $x\to\infty$. So, $\P(S_n>x)\OG n\P(X_1>x)\asymp nx^{-\al}$, by \eqref{eq:p1}. Thus, $\P(S_n>x)\to0$ implies \eqref{eq:LD}, which in turn implies 
$\P(S_n>x)\asymp nx^{-\al}\to0$. 
\end{proof}

\def\cprime{$'$} \def\polhk#1{\setbox0=\hbox{#1}{\ooalign{\hidewidth
  \lower1.5ex\hbox{`}\hidewidth\crcr\unhbox0}}}
  \def\polhk#1{\setbox0=\hbox{#1}{\ooalign{\hidewidth
  \lower1.5ex\hbox{`}\hidewidth\crcr\unhbox0}}}
  \def\polhk#1{\setbox0=\hbox{#1}{\ooalign{\hidewidth
  \lower1.5ex\hbox{`}\hidewidth\crcr\unhbox0}}} \def\cprime{$'$}
  \def\polhk#1{\setbox0=\hbox{#1}{\ooalign{\hidewidth
  \lower1.5ex\hbox{`}\hidewidth\crcr\unhbox0}}} \def\cprime{$'$}
  \def\polhk#1{\setbox0=\hbox{#1}{\ooalign{\hidewidth
  \lower1.5ex\hbox{`}\hidewidth\crcr\unhbox0}}} \def\cprime{$'$}
  \def\cprime{$'$}

\nopagebreak
\nopagebreak
\begin{affil}
Department of Mathematical Sciences, 
Michigan Technological University, 
Houghton, Michigan 49931\\
ipinelis@mtu.edu
\end{affil}

\end{document}